\documentclass[a4paper, 12pt]{amsart}

\usepackage[T1]{fontenc}
\usepackage[utf8]{inputenc}

\usepackage{amsmath, amssymb, amsthm}
\usepackage{chngcntr, cite, color}
\usepackage{dsfont}
\usepackage{enumitem, etoolbox}
\usepackage{hyperref}
\usepackage{indentfirst}

\hypersetup{colorlinks = true, linkcolor = black, filecolor = black, urlcolor = black, citecolor = black}

\makeatletter

    \renewcommand{\l@section}{\@tocline{1}{0pt}{1.5pc}{5pc}{}}
    \renewcommand{\l@subsection}{\@tocline{2}{0pt}{3.4pc}{0pc}{}}

    \newtheoremstyle{italicremark}
    {}{}
    {}
    {}
    {\itshape}
    {.}
    { }
    {\thmname{#1}\thmnumber{ #2}\thmnote{ (#3)}}
\makeatother

\newcounter{proofstep}
\AtBeginEnvironment{proof}{\setcounter{proofstep}0}

\newcounter{proofcase}
\AtBeginEnvironment{proof}{\setcounter{proofcase}0}

\newtheorem{theorem}{Theorem}[section]

\newtheorem{corollary}[theorem]{Corollary}
\newtheorem{lemma}[theorem]{Lemma}

\theoremstyle{definition}

\theoremstyle{italicremark}
\newtheorem{remark}[theorem]{\textit{Remark}}

\newcommand{\Rbb}{\mathbb{R}}

\newcommand{\revnew}[1]{{\color{red}\ifmmode\else\hypersetup{citecolor=red}\fi#1}}

\numberwithin{equation}{section}

\title[Conformally flat free boundary minimal hypersurfaces]{Conformally flat free boundary minimal hypersurfaces in the ball}
\author[Roney Santos]{Roney Santos}

\address{Roney Santos \newline
\indent \normalfont Instituto de Matem\'atica e Estat\'istica \newline
\indent Universidade de São Paulo \newline
\indent Rua do Mat\~ao, 05508-900, S\~ao Paulo-SP, Brasil \newline
\indent \texttt{roneypsantos@ime.usp.br}}
	

\begin{document}

\begin{abstract}
    In this note, we show that a locally conformally flat free boundary minimal hypersurface in $B^n$ must be either an equatorial ball or a critical catenoid for all $n \geq 4$. As a consequence, we conclude that the equatorial ball is unique among locally conformally flat free boundary minimal hypersurfaces in $B^n$ with the topology of $B^{n - 1}$.
\end{abstract}

\maketitle

\section{Introduction}

A properly immersed minimal submanifold in the unit Euclidean ball $B^n$ that is smooth and connected, and whose boundary is smooth and meets $\partial B^n$ orthogonally is commonly referred to as \emph{free boundary minimal submanifold} in $B^n$. The simplest example of such submanifolds is the \emph{equatorial ball}, namely a piece of vector subspace inside the ball. In codimension one, a piece of minimal non-planar rotational hypersurface intersecting $B^n$ and meeting $\partial B^n$ orthogonally is a \emph{critical catenoid}.\smallskip

Our goal is to show that these are the only examples in codimension one that are conformally flat.

\begin{theorem}\label{thm:main}
    Let $\Sigma$ be a locally conformally flat free boundary minimal hypersurface in $B^n$ with $n \geq 4$. Then $\Sigma$ is congruent to either the equatorial ball or the critical catenoid.
\end{theorem}

The theory of conformally flat minimal hypersurfaces is a traditional subject in differential geometry. In $\Rbb^n$, when $n \geq 5$ it is well-known since Cartan \cite{cartan2917conforme} that one of the principal curvatures of such hypersurfaces has multiplicity at least $n - 2$, and a complete classification was independently obtained by Blair \cite{blair1975catenoid} and Pinl-Ziller \cite{pinl1976hypersurfaces} (see also do Carmo-Dajczer \cite{docarmo1983rotation}). So, our Theorem \ref{thm:main} in dimensions $n \geq 5$ is a trivial consequence of this classification result, once the critical catenoid is unique among non-planar rotational hypersurfaces with free boundary in the ball (see \cite{freidin2017cohomogeneity, smith2021index}).\smallskip

The general scenario becomes much more complicated if $n = 4$, in which case there are conformally flat minimal hypersurfaces with three different principal curvatures (see, for instance, \cite{dajczer2019minimalconformallyflat}). Despite this, using the system of partial differential equations for holonomic hypersurfaces, a full local classification was obtained by Rei-Tojeiro \cite{dajczer2019minimalconformallyflat}. However, our Theorem \ref{thm:main} does not follow directly from their classification, once, in principle, we cannot avoid the examples with three distinct principal curvatures.\smallskip

Hence Theorem \ref{thm:main} implies that the free boundary condition is strong enough to both generalize the rotational behaviour of conformally flat minimal hypersurfaces of $\Rbb^n$ in high dimensions to $n = 4$, and to rule out the examples with three distinct principal curvatures.\smallskip

One direct consequence of Theorem \ref{thm:main} reads as follows.

\begin{corollary}\label{cor:main}
    Let $\Sigma$ be a locally conformally flat free boundary minimal hypersurface in $B^n$ with $n \geq 4$. Suppose $\Sigma$ is homeomorphic to $B^{n - 1}$. Then $\Sigma$ is congruent to the equatorial ball.
\end{corollary}

The well-known Nitsche theorem \cite{nitsche1985stationary} asserts that a free boundary minimal disk in $B^3$ must be flat. It was later generalized in the works of Fraser-Schoen \cite{fraser2025uniqueness} for free boundary minimal disks in the ball of arbitrary dimension, and Wheeler-Wheeler \cite{wheeler2019minimalhypersurfaces} for free boundary minimal balls in $B^n$ that are graphical with respect to a given Killing vector field. Our Corollary \ref{cor:main} may also be viewed as a Nitsche-type theorem under the conformal assumption.\smallskip

Finally, we point out that Theorem \ref{thm:main} cannot be extended to dimension $n = 3$ since every surface is locally conformally flat, and there are plenty of examples of free boundary minimal annuli in $B^3$ that cannot be rotational as shown in \cite{fernandez2023immersed, kapouleas2022free}.

\subsubsection*{LLM disclosure}

This work does not contain material written by AI, although we have used ChatGPT to review definitions and results already known in the literature, as well as to perform some computations. All mathematical arguments were independently checked by the author, who assumes full responsibility for the final manuscript.

\subsubsection*{Acknowledgments}

The author is thankful to Ezequiel Barbosa for suggesting this problem, and to Detang Zhou for pointing out that using the Schouten tensor was the most natural way to treat it. This work was supported by FAPESP grant 2023/14796-3.

\section{Proof of Theorem \ref{thm:main}}

We start with the following result showing that rotational symmetry imposed along one boundary component of a free boundary minimal submanifold in $B^n$ propagates to the whole submanifold.

\begin{lemma}\label{lem:rotational}
    Let $\Sigma$ be a free boundary minimal submanifold in $B^n$. Suppose that at least one component of $\partial \Sigma$ is $G$-invariant under a connected subgroup $G \subset SO(n)$. Then $\Sigma$ is $G$-invariant.
\end{lemma}

\begin{proof}
    Let $x: M \to B^n$ be an immersion with $x(M) = \Sigma$, and let $\Gamma$ be a component of $\partial \Sigma$ invariant under $G \subset SO(n)$. For each Killing vector field $X$ generating a one-parameter subgroup of $G$, we have that $X$ is tangent to $\Gamma$ along it and that $X^\perp$ is Jacobi by minimality of $\Sigma$. Now, let $\nu = x$ denote the unit conormal of $\partial \Sigma$. Since $X(x) = Q(x)$ for some $Q \in \mathfrak{so}(n)$,
    \[\bar{\nabla}_\nu X = Q(\nu) = Q(x) = X(x),\]
    where $\bar{\nabla}$ is the connection of $\Rbb^n$, and so free boundary condition gives that along $\Gamma$
    \[\nabla^\perp_\nu X^\perp = \bigl(\bar{\nabla}_\nu X\bigr)^\perp - \bigl(\bar{\nabla}_\nu X^\top\bigr)^\perp = X^\perp - \alpha_\Sigma(X, \nu) = 0,\]
    where $\nabla^\perp$ is the normal connection of $\Sigma$ and $\alpha_\Sigma$ is the vectorial second fundamental form of $\Sigma$ in $\Rbb^n$. Therefore $X^\perp$ satisfies
    \[\left\{
        \begin{aligned}
            L(X^\perp) &= 0\, \mbox{ in }\, \Sigma\\
            X^\perp &= 0\, \mbox{ in }\, \Gamma\\
            \nabla_\nu^\perp X^\perp &= 0\, \mbox{ in }\, \Gamma,
        \end{aligned}
    \right.\]
    where $L$ is the Jacobi operator on the normal bundle of $\Sigma$. It follows by unique continuation that $X^\perp$ vanishes identically, which means that $X$ is tangent to $\Sigma$, so that $G(\Sigma) = \Sigma$.
\end{proof}

We now prove our main result in the low dimensional setting.

\begin{theorem}
    Let $\Sigma$ be a locally conformally flat free boundary minimal hypersurface in $B^4$. Then $\Sigma$ is congruent to either the equatorial ball or the critical catenoid.
\end{theorem}

\begin{proof}
    Let $x: M \to B^4$ be an immersion such that $x(M) = \Sigma$, and let $\nu = x$ be the outward unit conormal of $\partial \Sigma$ in $\Sigma$. Choose a local unit normal vector field $N$ along it, and define the scalar second fundamental form of $\Sigma$ in $\Rbb^4$ as
    \[A_\Sigma(X, Y) = \langle \bar{\nabla}_XY, N\rangle,\]
    where $\bar{\nabla}$ is the connection of $\Rbb^4$. By free boundary condition, we know that $\nu$ is a principal direction of $\Sigma$ along $\partial \Sigma$. Fix $p \in \partial \Sigma$, choose an orthonormal basis $\{e_1, e_2\}$ of $T_p(\partial \Sigma)$ that diagonalizes $A_\Sigma|_{T_p(\partial \Sigma) \times T_p(\partial \Sigma)}$ and set
    \[\lambda_1 = A_\Sigma(e_1, e_1),\quad \lambda_2 = A_\Sigma(e_2, e_2),\quad \mu = A_\Sigma(\nu, \nu).\]
    Let $S$ denote the Schouten tensor of $\Sigma$. By minimality of $\Sigma$,
    \[S = -\, A_\Sigma^2 + \frac{|A_\Sigma|^2}{4}g,\]
    where $g$ is the induced metric of $\Sigma$, and then $S(\nu, e_1) = S(\nu, e_2) = 0$ by free boundary condition. Consider an extension of $e_1$ and $e_2$ in $\partial \Sigma$ such that they are geodesic at $p$ with respect to the induced metric on $\partial \Sigma$. On one hand, since $e_i$ is unit, we have $\nabla_{e_i}e_i = -\, \nu$, where $\nabla$ is the connection induced on $\Sigma$, and thus
    \begin{equation}\label{eq:schouten1}
        \begin{aligned}
            (\nabla_{e_i}S)(\nu, e_i) &= e_i\big(S(\nu, e_i)\big) - S(\nabla_{e_i}\nu, e_i) - S(\nu, \nabla_{e_i}e_i)\\
            &= -\, S(e_i, e_i) + S(\nu, \nu) = \lambda_i^2 - \mu^2.
        \end{aligned}
    \end{equation}
    On the other hand, extending $\{e_1, e_2, \nu\}$ with $\nabla_\nu e_1 = \nabla_\nu e_2 = \nabla_\nu \nu = 0$ at $p$, one has
    \begin{equation}\label{eq:schouten2}
        (\nabla_\nu S)(e_i, e_i) = \nu\big(S(e_i, e_i)\big) - 2S(\nabla_\nu e_i, e_i) = \nu\Bigl(-\, \lambda_i^2 + \frac{|A_\Sigma|^2}{4}\Bigr).
    \end{equation}
    Moreover, again at $p$, after extending $\lambda_1$, $\lambda_2$ and $\mu$ using the above extension of $\{e_1, e_2, \nu\}$,
    \begin{equation}\label{eq:sff1}
        \frac{1}{2}\nu(|A_\Sigma|^2) = \lambda_1\, \nu(\lambda_1) + \lambda_2\, \nu(\lambda_2) + \mu\, \nu(\mu).
    \end{equation}
    Since $A_\Sigma$ is Codazzi,
    \begin{equation}\label{eq:sff2}
        \begin{aligned}
            \nu(\lambda_i) &= (\nabla_\nu A_\Sigma)(e_i, e_i) = (\nabla_{e_i} A_\Sigma)(e_i, \nu) \\
            &= e_i\big(A_\Sigma(e_i, \nu)\big) - A_\Sigma(\nabla_{e_i} e_i, \nu) - A_\Sigma(e_i, \nabla_{e_i} \nu) \\
            &= \mu - \lambda_i
        \end{aligned}
    \end{equation}
    and then by minimality of $\Sigma$,
    \[(\nabla_\nu A_\Sigma)(\nu, \nu) = -\, (\nabla_\nu A_\Sigma)(e_1, e_1) - (\nabla_\nu A_\Sigma)(e_2, e_2) = -\, 3\mu.\]
    As a consequence, equation \eqref{eq:sff1} becomes
    \begin{equation}\label{eq:sff3}
        \nu(|A_\Sigma|^2) = -\, 2(|A_\Sigma|^2 + 3\mu^2).
    \end{equation}
    Plugging equations \eqref{eq:sff2} and \eqref{eq:sff3} into \eqref{eq:schouten2}, we get
    \begin{equation}\label{eq:schouten3}
        (\nabla_\nu S)(e_i, e_i) = 2\lambda_i(\lambda_i - \mu) - \frac{1}{2}(|A_\Sigma|^2 + 3\mu^2).
    \end{equation}
    Since $S$ is Codazzi, equations \eqref{eq:schouten1} and \eqref{eq:schouten3} must coincide. In particular,
    \[(\lambda_i - \mu)^2 = \frac{1}{2}(|A_\Sigma|^2 + 3\mu^2),\]
    which shows that $(\lambda_1 - \mu)^2 = (\lambda_2 - \mu)^2$. Thus, by minimality of $\Sigma$,
    \begin{equation}\label{eq:zero1}
        \mu(\lambda_1 - \lambda_2) = 0.
    \end{equation}
    Note that, by free boundary condition, the second fundamental form of $\partial \Sigma$ as a surface in $S^3$ is just $A_{\partial \Sigma} = A_\Sigma|_{T(\partial \Sigma) \times T(\partial \Sigma)}$, which means that $\lambda_1$ and $\lambda_2$ are principal curvatures of $\partial \Sigma$ in $S^3$ as well, so that the mean curvature of $\partial \Sigma$ in $S^3$ is $H_{\partial \Sigma} = -\, \mu$. Hence \eqref{eq:zero1} implies that $\partial \Sigma$ is the union of pieces of minimal and umbilical surfaces of $S^3$.\smallskip

    Let $\Gamma$ be a connected component of $\partial \Sigma$. Assume that $H_{\partial \Sigma}$ vanishes identically along it. Since locally $x^\perp = \langle x, N \rangle N$, minimality of $\Sigma$ implies that $x^\perp$ is Jacobi. Moreover, on $\Gamma$ free boundary condition shows
    \[\nabla^\perp_\nu x^\perp = \nu \langle x, N\rangle N + \langle x, N\rangle \nabla^\perp_\nu N = \bigl(\langle \bar{\nabla}_\nu\, x, N\rangle + \langle x, \bar{\nabla}_\nu N\rangle\bigr)N.\]
    However,
    \[\langle \bar{\nabla}_\nu\, x, N\rangle = \langle \nu, N \rangle = 0\]
    and on $\Gamma$
    \[\langle x, \bar{\nabla}_\nu N\rangle = -\, A_\Sigma(\nu, \nu) = -\, \mu = H_{\partial \Sigma} = 0.\]
    Therefore,
    \[\left\{
        \begin{aligned}
            L(x^\perp) &= 0\, \mbox{ in }\, \Sigma\\
            x^\perp &= 0\, \mbox{ in }\,\partial\Sigma\\
            \nabla^\perp_\nu x^\perp &= 0\, \mbox{ in }\, \Gamma,
        \end{aligned}
    \right.\]
    where $L$ is the Jacobi operator on the normal bundle of $\Sigma$. By unique continuation $x^\perp$ vanishes identically, and then $\Sigma$ is a flat ball.\smallskip

    If $H_{\partial \Sigma}(p) \neq 0$ for some $p \in \Gamma$, there exists a connected open subset $U \subset \Gamma$ containing $p$ such that $\lambda_1 = \lambda_2 = \lambda$ on it, so that it is totally umbilical in $S^3$. By the Codazzi equations, $\lambda$ is constant on $U$, and by continuity it is also constant in $\bar{U}$. It follows that if $H_{\partial \Sigma}$ is zero at some point of $\partial U$, it would vanish on the whole of $U$, contradicting its choice. Consequently, by taking $U$ maximal, one gets $\partial U = \emptyset$, and thus $U = \Gamma$ is totally umbilical in $S^3$. This means that $\Gamma$ is a round two-sphere, and therefore $\Sigma$ is rotational by Lemma \ref{lem:rotational}. Since $\Sigma$ cannot be an equatorial ball in that case, the uniqueness of the critical catenoid among non-planar rotational free boundary minimal hypersurfaces of the ball concludes our assertion.

\end{proof}

\begin{remark}
    Our computations above can be applied also to the case $n \geq 5$ in the following way. Let $\lambda_1, \ldots, \lambda_{n - 2}, \mu$ be the principal curvatures of $\Sigma$ along $\partial \Sigma$, where $\mu = A_\Sigma(\nu, \nu)$. As before, one shows that equation \eqref{eq:zero1} takes the general form
    \begin{equation}\label{eq:zero2}
        (\lambda_i - \lambda_j)(\lambda_i + \lambda_j - 2\mu) = 0
    \end{equation}
    for all $i \neq j$. Since one of the principal curvatures of $\Sigma$ has multiplicity at least $n - 2$, for a given boundary component $\Gamma \subset \partial \Sigma$ we have that either $\lambda_1 = \ldots = \lambda_{n - 2}$ or $\lambda_2 = \ldots = \lambda_{n - 2} = \mu$ at each $p \in \Gamma$. In this second case \eqref{eq:zero2} implies $\lambda_1 = \mu$. Since the second fundamental form of $\partial \Sigma$ inside of $S^{n - 1}$ coincides with $A_\Sigma|_{T(\partial \Sigma) \times T(\partial \Sigma)}$ by free boundary condition, one concludes that in either case $\Gamma$ is totally umbilical in $S^{n - 1}$, which means that $\Gamma$ is a round $(n - 2)$-sphere. The result now follows from Lemma \ref{lem:rotational} as before.
\end{remark}

\bibliographystyle{siam}
\bibliography{references}
	
\end{document}